\documentclass[11pt]{amsart}

\newtheorem{theorem}{Theorem}[section]

\newtheorem{corollary}[theorem]{Corollary}

\newtheorem{lemma}[theorem]{Lemma}

\newtheorem{proposition}[theorem]{Proposition}

\theoremstyle{definition}
\newtheorem{remark}[theorem]{Remark}

\usepackage[colorlinks, bookmarks=true]{hyperref}
\usepackage{color,graphicx,shortvrb}

\usepackage{enumerate}
\usepackage{amsmath}
\usepackage{amssymb}

\def\NN{{\mathbb{N}}}
\def\11{\textbf{$1$}}

\newcommand{\mx}{\mathrm{max}}
\newcommand{\mn}{\mathrm{min}}

\newcommand{\one}{\mathbf{1}}

\newcommand{\rank}{\mathrm{rank}}
\newcommand{\spn}{\mathrm{span}}
\newcommand{\vr}{\varepsilon}

\newcommand{\ball}{\mathbf{B}}

\newcommand{\injtens}{\check{\otimes}}
\newcommand{\projtens}{\hat{\otimes}}

\newcommand{\kwap}{{\mathrm{Kw}}}

\begin{document}

\numberwithin{equation}{section}

\title[Injectivity and projectivity in $p$-multinormed spaces]{Injectivity and projectivity in $p$-multinormed spaces}

\author{T.~Oikhberg}

\address{ Dept.~of Mathematics, University of Illinois, Urbana IL 61801, USA}
\email{oikhberg@illinois.edu}

\date{July 7, 2017}



\maketitle

\parindent=0pt
\parskip=3pt

\begin{abstract}
We find large classes of injective and projective $p$-multinormed spaces.
In fact, these classes are universal, in the sense that every $p$-multinormed
space embeds into (is a quotient of) an injective (resp.~projective) $p$-multinormed space.
As a consequence, we show that any $p$-multinormed space has a canonical
representation as a subspace of a quotient of a Banach lattice.
\end{abstract}

\maketitle
\thispagestyle{empty}

\section{Introduction}\label{s:intro}

The study of $p$-multinormed spaces has its roots in the early 1990s,
see \cite{MaN}, or the sadly unpublished thesis \cite{McC}.
In recent years, significant progress has been achieved.
For instance, multinorms are employed in \cite{DP} to investigate
homological properties of Banach modules over group algebras.
In \cite{CN}, multinorms are used to describe properties of Banach spaces (such as $GL_2$)
in terms of ``special'' embeddings into Banach lattices.
Many common constructions, known for Banach and operator spaces (subspaces, quotients,
duality, etc.) have been described in the $p$-multinormed case as well.
We refer the reader to the recent paper \cite{DLOT} for an introduction to the topic.

However, the injective and projective objects in this category have not been
described (with the exception of $p=\infty$). In this paper, we exhibit wide classes
of injective and projective $p$-multinormed spaces, and show that
these classes are universal: every $p$-multinormed
space embeds into an injective object, and is a quotient of a projective one.
Note that similar results are well known in the setting of Banach and operator spaces.
Sections \ref{s:inj mult} and \ref{s:projectivity} deal with the injective and projective
settings, respectively (see the relevant definitions in Section \ref{s:basic}).

The theory of $p$-multinormed spaces was partially motivated by giving an abstract
description of (subspaces of) Banach lattices. Indeed, a Banach lattice $X$ can be
equipped with its natural $p$-multinorm (described below).
By \cite{McC}, any $\infty$-multinormed space embeds into a Banach lattice.
By \cite{DLOT}, for general $p$, the same is true with the extra assumption of
$p$-convexity (and fails without $p$-convexity).
In Theorem \ref{t:subquotient}, we show that any $p$-multinormed space has a canonical
representation as a subspace of a quotient of a Banach lattice.

We should also mention Theorem \ref{t:control} which shows that, for every $\vr > 0$ and
$n \in \NN$, there exists $M = M(n,\vr)$ so that, for any rank $n$ operator $T$
between $p$-multinormed spaces, we have $\|T\|_p \leq (1+\vr) \|I_{\ell^p_M} \otimes T\|$.
This contrasts with the possible lack of exactness of operator spaces
(see e.g. \cite{ER}).


The standard notation is used throughout the paper. We denote the closed unit ball
of a normed space $Z$ by $\ball(Z)$. In Section \ref{s:basic} below, we
review some relevant facts and definitions regarding multinormed spaces and
maps between them (see \cite{DLOT} for proofs and further information).
Unless stated otherwise, we deal with spaces over the real field.

\section{Preliminaries on $p$-multinormed spaces}\label{s:basic}

First review some relevant facts and definitions (see e.g.~\cite{DLOT}).
A {\emph{$p$-multinormed space}} ($1 \leq p \leq \infty$)
is a Banach space $X$, together with a left tensorial cross-norm
$\| \cdot \|_\alpha$ on $\ell^p \otimes X$
(when $p = \infty$, we consider $c_0$ instead of $\ell^\infty$).
Recall that such a norm must satisfy the following conditions:
\begin{enumerate}
\item
For $a \in \ell^p$ and $x \in X$, $\|a \otimes x\|_\alpha = \|a\| \|x\|$;
\item
For $u \in \ell^p \otimes X$, $\|u\|_\alpha \geq \sup | \langle a^* \otimes x^* , u \rangle |$,
where the supremum runs over all $a^* \in (\ell^p)^*, x^* \in X^*$ of norm not exceeding $1$;
\item
For $T \in B(\ell^p)$ and $u \in \ell^p \otimes X$,
$\|T\| \|u\|_\alpha \geq \|(T \otimes I_X)(u)\|_\alpha$
\end{enumerate}
(conditions (1) and (2) define a cross-norm; (3) describes being left tensorial).
We refer the reader to \cite{DFS} or \cite{Rya} for more information about tensor products.

Viewing the spaces $\ell^p_n$ as coordinate subspaces of $\ell^p$ (or $c_0$, for $p = \infty$),
we denote the restriction of $\| \cdot \|_\alpha$ to $\ell^p_n \otimes X$ by $\| \cdot \|_n$
(or simply $\| \cdot \|$, if confusion is unlikely).
The sequence of norms $(\| \cdot \|_n)$ is sometimes referred to as the
{\emph{$p$-multinorm structure}} on $X$. The norms $\| \cdot \|_n$ are cross-norms, and, for any
$T \in B(\ell^p_n, \ell^p_m)$ and $x \in \ell^p_n \otimes X$, we have
$\|T\| \|x\|_n \geq \|(T \otimes I_X)(x)\|_m$ (we can say that $(\| \cdot \|_n)_{n \in \NN}$
is a ``tensorial sequence of norms''). From 
any tensorial sequence of cross-norms on $\ell^p_n \otimes X$,
we can reconstruct a norm on $\ell^p \otimes X$
($c_0 \otimes X$, for $p = \infty$) satisfying $(1)$, $(2)$, and $(3)$. 
This observation allows us to only consider the spaces $(\ell^p_n \otimes X, \| \cdot \|_n)$.

If $X$ and $Y$ are $p$-multinormed spaces, then a linear map $u : X \to Y$
is \emph{$p$-multibounded} if its \emph{$p$-norm}
$$
\|u\|_p = \sup_n \|I_{\ell^p_n} \otimes u : \ell^p_n \otimes X \to \ell^p_n \otimes Y\| =
\|I_{\ell^p} \otimes u : \ell^p \otimes X \to \ell^p \otimes Y\| 
$$
is finite (as before, for $p = \infty$ we replace $\ell^\infty$ by $c_0$).
We denote by $M_p(X,Y)$ the set of $p$-multibounded operators from $X$ to $Y$.
Clearly every $u \in M_p(X,Y)$ is bounded, with $\|u\| \leq \|u\|_p$.
The equality holds when either $X$ or $Y$ is $1$-dimensional.
In general, a bounded operator need not be $p$-multibounded.
An operator $u$ is called a $p$-multicontraction if $\|u\|_p \leq 1$.

Duality pairing exists between $p$-multinormed and $p^\prime$-multinormed spaces
(here and throughout the paper, $p$ and $p^\prime $ are ``conjugate:'' $1/p + 1/p^\prime = 1$).
Suppose $X$ is a $p$-multinormed space, with underlying Banach space $X^\prime$
(usually, we use the same notation for the $p$-multinormed space and for its Banach
space, but in this paragraph, we find it convenient to use the ``forgetful functor'' $^\prime$).
For $n \in \NN$ denote by $(\delta_i)_{i=1}^n$ and $(\delta_i^\prime)_{i=1}^n$
the canonical bases of $\ell^p_n$ and $\ell^{p^\prime}_n$, respectively.
For $x_1^*, \ldots, x_n^* \in X^{\prime *}$ we set
$$
\big\| \sum_{i=1}^n \delta_i^\prime \otimes x_i^* \big\| =
\sup \Big\{ \big| \sum_{i=1}^n \langle x_i^*, x_i \rangle \big| :
 \big\| \sum_{i=1}^n \delta_i^\prime \otimes x_i \big\| \leq 1 \Big\} .
$$
In other words, $\ell^{p^\prime}_n \otimes X^*$ is the dual of $\ell^p_n \otimes X$.
Clearly this defines a $p^\prime$-multinormed space $X^*$, ``built on'' $X^{\prime *}$.

Several standard $p$-multinorm structures will be used in this paper.
Suppose $E$ is a Banach space. The \emph{maximal} and \emph{minimal} $p$-multinorms
correspond to 
$\ell^p_n \projtens E$ and $\ell^p_n \injtens E$,
where $\projtens$ and $\injtens$ refer to the Banach space projective and
injective tensor products. These structures will be denoted by $\mx_p(E)$ and $\mn_p(E)$
respectively. The origin of this terminology is transparent: if $Z$ is a $p$-multinormed space,
then for any $u  \in B(Z,\mn_p(E))$ and $v \in B(\mx_p(E),Z)$, we have
$\|u\|_p = \|u\|$ and $\|v\|_p = \|v\|$.
Duality functions in the expected way: for any Banach space $E$,
$(\mn_p(E))^* = \mx_{p^\prime}(E^*)$, and $(\mx_p(E))^* = \mn_{p^\prime}(E^*)$.

On a Banach lattice $E$, we can define its \emph{natural $p$-multinorm}:
if $\delta_1, \ldots, \delta_n$ is the canonical basis of $\ell^p_n$, and
$x_1, \ldots, x_n \in E$, we set
$$
\big\| \sum_i \delta_i \otimes x_i\big\|_{\ell^p_n \otimes E} =
 \Big\| \big( \sum_i |x_i|^p \big)^{1/p} \Big\|_E 
$$
(the right hand side is defined using the Krivine functional calculus, see e.g.~\cite{M-N}).
One can observe (see \cite[p.~21]{MaN}) that, if $E, F$ are Banach lattices, and
$u : E \to F$ is a positive operator, then $\|u\|_p = \|u\|$.

For $p$-multinormed spaces $(X_i)_{i \in I}$, we can define their
$\ell^\infty$, $c_0$, and $\ell^1$ sums. 
To this end, denote by $(X_i^\prime)_{i \in I}$ the underlying Banach spaces
(once more, we formally distinguish between $p$-multinormed spaces,
and the corresponding Banach spaces).

If $Y = (\sum_i X_i^\prime)_\infty$, then any element
of $\ell^p_n \otimes Y$ can be written as $(x_1, x_2, \ldots)$, with
$x_i \in \ell^p_n \otimes X_i^\prime$. Then set
$\|x\|_n = \sup_i \|x_i\|_{\ell^p_n \otimes X_i}$.
Clearly this defines a $p$-multinormed space (the $\ell^\infty$ sum,
denoted by $(\sum_i X_i)_\infty$), with $Y$ its underlying Banach space.
The $c_0$ sum, denoted by $(\sum_i X_i^\prime)_0$, is defined in a similar fashion.

To define the $\ell^1$ sum, consider a family $x_i \in \ell^p_n \otimes X_i$ ($i \in I$),
with finitely many nonzero entries. For
$x = (x_i)_{i \in I} \in \ell^p_n \otimes (\sum_i X_i^\prime)_{00}$, set
$\|x\| = \sum_i \|x_i\|$.
The completion of $\ell^p_n \otimes (\sum_i X_i^\prime)_{00}$ in this norm is denoted
by $(\sum_i X_i)_1$; clearly the underlying Banach space is $(\sum_i X_i^\prime)_1$.

It is easy to see that, for Banach spaces $(E_i)_{i \in I}$,
$$
\mn_p \Big( \big( \sum_i E_i \big)_\infty \Big) =
\Big( \sum_i \mn_p \big( E_i \big) \Big)_\infty ,
$$
$$
\mn_p \Big( \big( \sum_i E_i \big)_0 \Big) =
\Big( \sum_i \mn_p \big( E_i \big) \Big)_0 ,
$$
and
$$
\mx_p \Big( \big( \sum_i E_i \big)_1 \Big) =
\Big( \sum_i \mx_p \big( E_i \big) \Big)_1 .
$$
Further, for any family of $p$-multinormed spaces $(X_i)_{i \in I}$,
$$
\Big( \big( \sum_i X_i \big)_1 \Big)^* = \big( \sum_i X_i^* \big)_\infty
{\textrm{   and   }}
\Big( \big( \sum_i X_i \big)_0 \Big)^* = \big( \sum_i X_i^* \big)_1 .
$$





We say that a $p$-multinormed space is \emph{$c$-injective} if, for any
$p$-multinormed spaces $X \subset Y$, 
any $p$-multibounded operator $u : X \to Z$ has an extension 
$\tilde{u} : Y \to Z$ with $\|\tilde{u}\|_p \leq c \|u\|_p$,
and $\tilde{u}|_X = u$. In the dual setting, $Z$ is called $p$-projective if,
for every $\vr > 0$, every $p$-quotient map $q : Y \to X$, and every
$p$-multibounded $u : Z \to X$, there exists a lifting $\tilde{u} : Z \to Y$,
so that $\|\tilde{u}\|_p \leq \|u\|_p + \vr$, and $q \tilde{u} = u$.

Similar notions of injectivity and projectivity for Banach and operator spaces
are well known. For instance, $\ell^\infty(I)$ ($\ell^1(I)$) is a $1$-injective
(resp.~$1$-projective) Banach space, for any index set $I$.
The famous Stinespring Extension Theorem asserts that $B(H)$ is a
$1$-injective operator space for any Hilbert space $H$. It is shown in \cite{Ble}
that $\big( \sum_i T(H_i) \big)_1$ is $1$-projective if $(H_i)_{i \in I}$ is a family
of finite dimensional Hilbert spaces ($T(H)$ denotes the trace class on $H$).
It follows that any Banach (operator) space embeds into an injective object of the 
appropriate category, and is a quotient of a projective object. The main goal of this
paper is to establish similar results for $p$-multinormed spaces.

To proceeds, we need to show that, for for finite rank operators, we may limit the
calculate the $p$-multinorm on tensor products with $\ell^p_n$ with limited $n$
(one can view this as an analogue of exactness for operator spaces).

\begin{theorem}\label{t:control}
Suppose $X, Y$ are $p$-multinormed spaces, and $T : X \to Y$ is a rank $n$ operator.
Fix $\vr \in (0,1)$, and let $M = 2^n \lceil 4n^3/\vr \rceil^n$. Then
$\|T\|_p \leq (1+\vr) \|I_{\ell^p_M} \otimes T\|$.
\end{theorem}

The following result is based on \cite{PR}.

\begin{lemma}\label{l:l_p f.d.}
Suppose $Z$ is an $n$-dimensional subspace of $L^p(\mu)$, $1 \leq p \leq \infty$.
Fix $\vr \in (0,1)$, and let $M_0 = M_0(n,\vr) := 2^n \lceil 2n^2/\vr \rceil^n$.
\begin{enumerate}
\item 
$L^p(\mu)$ contains a sublattice $E$, spanned by characteristic functions
of disjoint sets, of dimension $M \leq M_0$,
so that, for every $z \in \ball(Z)$ there
exists $y \in E$, with $\|z - y\| \leq \vr$.
\item
Moreover, $E$ is lattice isometric to $\ell^p_M$, and
there exists a positive contractive projection $P$ from $L^p(\mu)$ onto $E$.
This projection has the property that, for every $z \in Z$, $\|Pz - z\| \leq 2\vr$.
\end{enumerate}
\end{lemma}

\begin{proof}
We present the proof for $1 \leq p < \infty$.
The case of $p=\infty$ requires only minor adjustments.

(1)
Let $z_1, \ldots, z_n$ be a normalized Auerbach basis in $Z$.
Set $z = \sum_i |z_i|$, and $z_0 = z/\|z\|$ (note that $\|z\| \leq n$).
Consider the probability measure $\nu = z_0^p \mu$,
concentrated on $S = \{s \in \Omega : z(s) \neq 0\}$.
The map $J : L^p(\mu) \to L^p(\nu)$, defined by setting
$$
[Jf](s) = \left\{ \begin{array}{ll}
   f(s)/z_0(s)  &   s \in S   \\   0  &  s \notin S   \end{array}  \right.
$$
is a lattice homomorphism, and moreover,
$\|Jf\|_{L^p(\nu)} = \|f \one_S\|_{L^p(\mu)}$.
Henceforth, we will work with the space $U = JZ$, instead of $Z$.
Note that the vectors $u_i = Jz_i$ form a normalized Auerbach basis in $U$.
Moreover, above $\|z\| \leq n$, hence $\max_i \|u_i\|_\infty \leq n$.

Set $k = \lceil 2n^2/\vr \rceil$, $A_j = [-n+2n(j-1)/k,-n+2nj/k)$ for $1 \leq j \leq k-1$,
and $A_k = [n-2n/k,n]$. For $1 \leq i \leq n$, set $B_{ij} = u_i^{-1}(A_j)$,
and define
$$
v_i = \sum_{j=1}^k n\Big(-1 + \frac{2j-1}{k}\Big) \one_{B_{ij}} .
$$
Then
$$
\|u_i - v_i\|_p \leq \|u_i - v_i\|_\infty \leq \frac{n}{k} \leq \frac{\vr}{2n} .
$$
Moreover, any norm one $u \in U$ can be expressed as $u = \sum_i \alpha_i u_i$,
with $\max |\alpha_i| \leq 1$, and therefore,
$$
\|u - \sum_i \alpha_i v_i\|_p \leq \|u - \sum_i \alpha_i v_i\|_\infty \leq
 \sum_i |\alpha_i| \|u_i - v_i\|_\infty \leq n \frac{\vr}{2n} = \frac{\vr}{2} .
$$
Let $S$ be the algebra of sets generated by the sets $B_{ij}$ ($1 \leq i \leq n$, $1 \leq j \leq k$).
As $B_{i\alpha} \cap B_{i\beta} = \emptyset$ whenever $\alpha \neq \beta$, the atoms of $S$
are of the form $\cap_{i=1}^n B_{i j_i}^{\vr_i}$, with $1 \leq j_i \leq k$, and
$B_{i j_i}^{\vr_i}$ being either $B_{i j_i}$, or its complement.
Consequently, $M = |S| \leq M_0 = (2k)^n$. Clearly the functions $v_i$
are constant on atoms of $S$.

(2)
Let $\mu^\prime$ be the restriction of $\mu$ to the set $S$.
Then $J$ is a surjective isometry from $L^p(\mu^\prime)$ onto $L^p(\nu)$.
Let $Q$ be the conditional expectation onto the sublattice
$F$, spanned by characteristic functions $S_1, \ldots, S_M$.
Finally, let $R$ be the restriction map from $L^p(\mu)$ to
$L^p(\mu^\prime)$ ($f \mapsto f \one_S$).
Then $P = J^{-1} Q J R$ is a contractive projective onto $E = J^{-1}(F)$.

Now consider $z \in \ball(Z)$. Then there exists $y \in E$ so that
$\|z - y\| \leq \vr$. By the triangle inequality,
$$
\|z - Pz\| = \|(I-P)(z-y)\| \leq 2 \|z-y\| \leq 2\vr ,
$$
which is what we need.
\end{proof}

\begin{lemma}\label{l:sublattice}
Suppose $Z$ is a $K$-dimensional sublattice of $\ell^p_N$, and $P$
a positive contractive projection from $\ell^p_N$ onto $Z$.
Then, for any operator $u : X \to Y$ between $p$-multinormed spaces, we have
$$
\|I_Z \otimes u : Z \otimes X \to Z \otimes Y\| = \|I_{\ell^p_K} \otimes u\| .
$$
Here we equip $Z \otimes X$ and $Z \otimes Y$ with the norm inherited from
$\ell^p_N \otimes X$, respectively $\ell^p_N \otimes Y$.
\end{lemma}

\begin{proof}
Let $V$ be the ``canonical'' isometry from $Z$ onto $\ell^p_K$ (mapping atoms to atoms).
We have to show that, for any $p$-multinormed space $E$, $v \otimes I_E$ is an isometry from
$Z \otimes E$ (equipped with the norm inherited from $\ell^p_m \otimes E$)
to $\ell^p_K \otimes E$. First, $v^{-1} : \ell^p_K \to Z \subset \ell^p_M$
is a positive isometry, hence, by \cite[p.~21]{MaN}, $\|v^{-1} \otimes I_E\| = \|v^{-1}\|$.
On the other hand,
$$
\|v \otimes I_E\| \leq \|vP \otimes I_E : \ell^p_N \otimes E \to \ell^p_K \otimes E\| =
\|vP\| \leq 1 .
\qedhere
$$
\end{proof}

\begin{proof}[Proof of Theorem \ref{t:control}]
Suppose first that $\dim X = n$. For simplicity, assume $\|I_{\ell^p_M} \otimes T\| \leq 1$
(then, in particular, $\|T\| \leq 1$).
Suppose $x$ is a norm one element of $\ell^p_N \otimes X$, and show that
$\|(I_{\ell^p_N} \otimes T)x\| \leq 1+\vr$.

Let $(x_i)_{i=1}^n$ be a normalized Auerbach basis in $X$, and write
$x = \sum_{i=1}^n a_i \otimes x_i$, with $a_i \in \ell^p_N$.
By the properties of the Auerbach basis, for every $i$ there exists a contractive projection
$P_i$ from $X$ onto $\spn[x_i]$, hence
$$
\|a_i\| = \|a_i \otimes x_i\| = \|(I_{\ell^p_N} \otimes P_i)x\| \leq \|x\| = 1 .
$$
Find a sublattice $Z \subset \ell^p_N$, of dimension $K \leq M$,
with the properties as in Lemma \ref{l:l_p f.d.}, so that, for any $i$,
there exists $\tilde{a}_i \in Z$ with $\|a_i - \tilde{a}_i\| \leq \vr/2n$.
Let $\tilde{x} = \sum_{i=1}^n \tilde{a}_i \otimes x_i$. By the triangle inequality,
\begin{equation}
\|(I_{\ell^p_N} \otimes T)x\| \leq
\|(I_{\ell^p_N} \otimes T) \tilde{x}\| +
\sum_{i=1}^n \|a_i - \tilde{a}_i\| \|T x_i\| .
\label{eq:find sum} 
\end{equation}
However, $\|T x_i\| \leq \|T\| \|x_i\| \leq 1$, hence
$\|\tilde{x}\| \leq \|x\| + \sum_i \|a_i - \tilde{a}_i\| \leq 1 + \vr/2$.
By Lemma \ref{l:sublattice},
$ \|(I_{\ell^p_N} \otimes T) \tilde{x}\| = \|(I_Z \otimes T) \tilde{x}\| =
\|(I_{\ell^p_K} \otimes T) \tilde{x}\| \leq \|\tilde{x}\|$.
Plugging all this back into \eqref{eq:find sum}, we obtain:
$$
\|(I_{\ell^p_N} \otimes T)x\| \leq 1 + \frac{\vr}{2} + n \frac{\vr}{2n} = 1 + \vr .
$$

In the general case, let $\tilde{X} = T/\ker T$. Use \cite[Section 1.3]{DLOT}
to write (in the canonical way) $T = \tilde{T} q$, with $\tilde{T} : \tilde{X} \to Y$,
and $q : X \to \tilde{X}$ being the quotient map.
Now apply the preceding reasoning to $\tilde{T}$.
\end{proof}

\section{Injectivity in $p$-multinormed spaces}\label{s:inj mult}


First we describe ``building blocks'' of $1$-injective objects.

\begin{proposition}\label{p:max p'}
For $1 \leq p \leq \infty$ and $n \in \NN$, $\mx_p(\ell^{p^\prime}_n)$
is $1$-injective as a $p$-multinormed space.
\end{proposition}

Show first that we can restrict our attention to finite dimensional spaces.

\begin{lemma}\label{l:crit inj}
Suppose $Z$ is a finite dimensional $p$-multinormed space
such that, for any finite dimensional $E \subset F$, and any
$u \in B(E,Z)$, there exists an extension $\tilde{u} : F \to Z$
with $\|\tilde{u}\|_p \leq \lambda \|u\|_p$. Then
$Z$ is $\lambda$-injective.
\end{lemma}

\begin{proof}
Suppose $X \subset Y$ are $p$-multinormed spaces, and $u : X \to Z$
satisfies $\|u\|_p \leq 1$. We have to show that $u$ has an extension
$\tilde{u} : Y \to Z$ with $\|\tilde{u}\|_p \leq \lambda$.

Note first that $X$ can be assumed to be finite dimensional.
Indeed, suppose the existence of an extension has been established
for finite dimensional $X$'s. For an arbitrary $X$,
let $X_0 = X/\ker u$ and $Y_0 = Y/\ker u$, and denote by $q$
the corresponding quotient map $Y \to Y_0$ (then $q$ is also
the quotient map from $X$ to $X_0$). Furthermore, $u$ generates
a $p$-multicontraction $u_0 : X_0 \to Z$, with $u = u_0 q$.
In turn, $u_0$ has an extension $\tilde{u}_0 : Y_0 \to Z$,
with $\|\tilde{u}_0\|_p \leq \lambda$. Now set $\tilde{u} = \tilde{u}_0 q$.

So, we can assume $X$ is finite dimensional. In $Y$, consider the
net of finite dimensional subspaces $F$, containing $X$, and ordered by inclusion.
Fix a $p$-multicontraction $u : X \to Z$. For each $F$ as above, $u$ has
an extension $\tilde{u}_F : F \to X$ with $\|\tilde{u}_F\|_p \leq \lambda$.
As $Z$ is finite dimensional, one can use a compactness argument to
achieve an extension $\tilde{u} : Y \to Z$.
\end{proof}

We need a technique for calculating $p$-multinorms
of maps $u : E \to \mx_p(\ell^{p^\prime}_n)$.
Below, we shall identify $e \in \ell^p_m \otimes E$ with an operator
${\mathbf{op}}(e) \in B(\ell^{p^\prime}_m, E)$.
Explicitly, if $(\delta_i)_{i=1}^n$ and $(\delta_i^\prime)_{i=1}^n$ are
the canonical bases in $\ell^p_n$ and $\ell^{p^\prime}_n$ respectively, and
$e = \sum_{i=1}^n \delta_i \otimes e_i$, then ${\mathbf{op}}(e)$ takes
$\delta_i^\prime$ to $e_i$.

\begin{lemma}\label{l:compute p-norm}
Suppose $E$ is a $p$-multinormed space. In the above notation,
$$
\big\| u : E \to \mx_p(\ell^{p^\prime}_n) \big\|_p =
\sup \Big\{ \big| {\mathrm{tr}} \big( u \circ {\mathbf{op}}(x) \big) \big| :
 x \in \ball(\ell^p_n \otimes E) \Big\} .
$$
holds for any $u : E \to \mx_p(\ell^{p^\prime}_n)$. Consequently,
$\|u\|_p = \|I_{\ell^p_n} \otimes u\|$.
\end{lemma}

\begin{proof}
By definition, 
\begin{equation}
\|u\|_p =
 \sup_{e \in \ball(\ell^p_m \otimes E)}
 \big\|(I_{\ell^p_m} \otimes u) e\big\|_{\ell^p_m \projtens \ell^{p^\prime}_n} =
 \sup_{e \in \ball(\ell^p_m \otimes E)} \nu_1 \big( u \circ {\mathbf{op}}(e) \big) ,
\label{eq:compute p norm}
\end{equation}
where $\nu_1( \cdot )$ refers to the nuclear norm of an operator.
As $\nu_1(w) \geq |{\mathrm{tr}}(w)|$ for every operator $w$
on a finite dimensional space,
$$
\big\| u : E \to \mx_p(\ell^{p^\prime}_n) \big\|_p \geq
\sup \Big\{ \big| {\mathrm{tr}}(u \circ {\mathbf{op}}(e)) \big| :
  \big\| e \big\|_{\ell^p_n \otimes E} \leq 1 \Big\} .
$$
It remains to prove the converse.
By the trace duality between $\nu_1( \cdot )$ and $\| \cdot \|$,
$$
\nu_1 \big( u \circ {\mathbf{op}}(e) \big) =
\sup \Big\{ \big| {\mathrm{tr}}(w \circ u \circ {\mathbf{op}}(e)) \big| \Big\} =
\sup \Big\{ \big| {\mathrm{tr}}(u \circ {\mathbf{op}}(e) \circ w) \big| \Big\} ,
$$
where the supremum runs over all contractions
$w : \ell^{p^\prime}_n \to \ell^{p^\prime}_m$.
We can view ${\mathbf{op}}(e) \circ w$ as
${\mathbf{op}}(\tilde{e}) \in B(\ell^{p^\prime}_n, E)$,
where $\tilde{e} = (w^* \otimes I_E) e$. Then
$$
\big\|\tilde{e}\big\|_{\ell^p_n \otimes E} \leq
\|w^*\| \|e\|_{\ell^p_m \otimes E} \leq 1 ,
$$
and therefore,
$$
\big\| u : E \to \mx_p(\ell^{p^\prime}_n) \big\|_p \leq
\sup \Big\{ \big| {\mathrm{tr}}(u \circ {\mathbf{op}}(e)) \big| :
  \big\| e \big\|_{\ell^{p_m} \otimes E} \leq 1 \Big\} .
$$

The preceding reasoning implies that, in \eqref{eq:compute p norm}, we can take the supremum
over all $e \in \ball(\ell^p_m \otimes E)$. This yields the last claim of the lemma.
\end{proof}

\begin{proof}[Proof of Proposition \ref{p:max p'}]
By Lemma \ref{l:crit inj}, it suffices to show that, for any pair of finite dimensional
$p$-multinormed spaces $X \subset Y$, and any $p$-mul\-ti\-con\-trac\-tion
$u : X \to \mx_p(\ell^{p^\prime}_n)$, there exists a $p$-mul\-ti\-con\-trac\-tive
extension $\tilde{u} : Y \to \mx_p(\ell^{p^\prime}_n)$.

By Lemma \ref{l:compute p-norm}, we can identify
$M_p(X, \mx_p(\ell^{p^\prime}_n))$ and $M_p(Y, \mx_p(\ell^{p^\prime}_n))$ with
the duals of $\ell^p_n \otimes X$ and $\ell^p_n \otimes Y$, respectively.
The former is a subspace of the latter, with the isometric embedding
implemented by $I_{\ell^p_n} \otimes j$ ($j$ being the formal identity from
$X$ to $Y$). Dualizing, we see that
$$
I_{\ell^{p^\prime}_n} \otimes j^* :
 \ell^{p^\prime}_n \otimes Y^* \to \ell^{p^\prime}_n \otimes X^*
$$
is a quotient map. Identifying the spaces above with
$M_p(Y, \mx_p(\ell^{p^\prime}_n))$ and $M_p(X, \mx_p(\ell^{p^\prime}_n))$ respectively,
we conclude that the operator
$$
M_p(Y, \mx_p(\ell^{p^\prime}_n)) \to M_p(X, \mx_p(\ell^{p^\prime}_n)) : v \mapsto vj
$$
is a quotient map. In other words, for any $u \in M_p(X, \mx_p(\ell^{p^\prime}_n))$, and
any $\vr > 0$, there exists $\tilde{u} \in M_p(Y, \mx_p(\ell^{p^\prime}_n))$, extending
$u$, and satisfying $\|\tilde{u}\|_p \leq (1+\vr) \|u\|_p$. As we are working with
finite dimensional spaces, the above result holds with $\vr = 0$ as well.
\end{proof}

\begin{corollary}\label{c:inj Lp}
If $(\Omega,\mu)$ is a $\sigma$-finite measure space, then $\mx_p(L^{p^\prime}(\mu))$
is $1$-injective as a $p$-multinormed space.
\end{corollary}

\begin{remark}\label{r:lea}
For $p=\infty$ this result is contained in \cite{McC}.
\end{remark}

\begin{proof}
Consider the family $I$ of all finite subalgebras of $\Omega$, ordered by inclusion
(if $\mu$ is a finite measure, we may additionally assume that they contain $\Omega$).
For $i \in I$ let $Z_i$ be the subspace of $L^{p^\prime}(\mu)$ generated by this algebra.
Clearly $Z_i$ is isometric to $\ell^{p^\prime}_{N_i}$, and moreover,
is the range of a conditional expectation $Q_i$
($Q_i$ is positive, hence, due to e.g. \cite[p.~21]{MaN} $p$-multicontractive).
Further, denote by $J_i$ the embedding of $Z_i$ into $L^{p^\prime}(\mu)$.
Note that, for any $z \in L^{p^\prime}(\mu)$, $\lim_i \|J_i Q_i z - z\| = 0$.

Suppose $X$ is a subspace of a $p$-multinormed space $Y$.
For a $p$-multicontraction $u : X \to L^{p^\prime}(\mu)$, we have to find
a $p$-multicontraction $\tilde{u} : Y \to L^{p^\prime}(\mu)$ extending it.
For $i \in I$, set $u_i = Q_i u : X \to Z_i$.
By the above, $u_i \to u$ in the point-norm topology.
By Proposition \ref{p:max p'}, $u_i$ has a $p$-multicontractive extension
$\tilde{u}_i : Y \to Z_i$. Then the net $(J J_i \tilde{u}_i)_{i \in I}$
($J$ is the canonical embedding of $L^{p^\prime}(\mu)$ into $L^{p^\prime}(\mu)^{**}$)
clearly belongs to the unit ball of
$B(Y, L^{p^\prime}(\mu)^{**}) = (Y \projtens L^{p^\prime}(\mu)^*)^*$
(see e.g. \cite[Section VIII.2]{DU}),
hence it has a weak$^*$ cluster point $v \in \ball(B(Y, L^{p^\prime}(\mu)^{**}))$. 

Now recall that there exists a positive (and therefore, $p$-multicontractive)
projection $Q$ from $L^{p^\prime}(\mu)^{**}$ onto $L^{p^\prime}(\mu)$.
Indeed, for $1 < p < \infty$ the identity map will do.
For $p = 1$, note that $L^1(\mu)^{**}$ is an abstract $L$-space.
By \cite[Section 2.4]{M-N}, $L^1(\mu)$ and its second dual are
KB-spaces (hence order continuous). By \cite[Theorem 2.4.10]{M-N},
$L^1(\mu)$ is a band in its second dual. The existence of $Q$
now follows from \cite[Corollary 2.4.4]{M-N}. To handle the case
of $p = \infty$, dualize.

We claim that $\tilde{u} = Q v$ has the desired properties --
that is, (i) $\tilde{u}|_X = u$, and (ii) $\|\tilde{u}\|_p \leq 1$.

For (i), pick $x \in X$. For any $z^* \in L^{p^\prime}(\mu)^*$,
$\langle x \otimes z^* , v \rangle = \langle z^*, vx \rangle$
is a cluster point of the net
$\langle x \otimes z^* , u_i \rangle = \langle z^*, u_i x \rangle$.
In other words, $v x$ is a weak$^*$ cluster point of the net $(u_i x)_{i \in I}$.
However, $u_i x \to u x$ in norm, hence $vx = ux$.

For (ii), pick $y = \sum_{j=1}^m \delta_j \otimes y_j \in \ball(\ell^p_n \otimes Y)$.
For any $i \in I$, $(I_{\ell^p_n} \otimes \tilde{u}_i) y =
 \sum_{j=1}^m \delta_j \otimes \tilde{u}_i y_j$ lies in the unit ball
of $\ell^p_n \projtens L^{p^\prime}(\mu)$. By duality, this is equivalent to the following:
if $\sum_{j=1}^n \delta_j^\prime \otimes z_j^*$ lies in the
unit ball of $\ell^{p^\prime}_n \injtens L^p(\mu)$, then
$|\sum_j \langle z_j^*, u_i y_j \rangle| \leq 1$.
As $v$ is a weak$^*$ cluster point of the net $(u_i)$,
$\sum_j \langle z_j^*, v y_j \rangle$ must be a cluster point
of the net $(\sum_j \langle z_j^*, u_i y_j \rangle)_i$.
Thus, $\|v\|_p \leq 1$.
\end{proof}

We immediately obtain:

\begin{theorem}\label{t:gen inj}
Suppose $(\Omega_i, \mu_i)_{i \in I}$ is a collection of measure spaces.
Then $(\sum_{i \in I} \mx_p(L^{p^\prime}(\mu_i)))_\infty$ is $1$-injective.
\end{theorem}

Next we show that any $p$-multinormed space embeds into
a $1$-injective $p$-multinormed space.

\begin{theorem}\label{t:embed inj}
Suppose $X$ is a $p$-multinormed space. Then there exists a family
of integers $(n_i)_{i \in I}$, so that $X$ embeds $p$-multiisometrically
into $(\sum_{i \in I} \mx_p(\ell^{p^\prime}_{n_i}))_\infty$.
If, moreover, $X$ is a dual $p$-multinormed space, then the embedding
can be made weak$^*$ to weak$^*$ continuous.
\end{theorem}

\begin{remark}
In \cite{McC}, the first statement of this proposition is proved for $p=\infty$
(using different methods).
\end{remark}

\begin{proof}
First show that, for any norm one $x \in \ell^p_n \otimes X$,
there exists a $p$-multicontraction $u : X \to \mx_p(\ell^{p^\prime}_n)$
so that
$$
\big\|(I_{\ell^p_n} \otimes u)(x)\big\|_{\ell^p_n \projtens \ell^{p^\prime}_n} = 1 .
$$
Note that $\ball(\ell^p_n \otimes X)$ is a convex balanced
subset of $\ell^p_n \otimes X$. By Hahn-Banach Theorem, there exists a linear
functional $x^* \in (\ell^p_n \otimes X)^*$ so that (i) $\langle x^*, x \rangle = 1$,
and (ii) $|\langle x^*, y \rangle| \leq 1$ for any $y \in \ball(\ell^p_n \otimes X)$.
Define ${\mathbf{op}}(x^*) , {\mathbf{op}}(y) : X \to \ell^{p^\prime}_n$
as in the paragraph preceding Lemma \ref{l:compute p-norm}.
By the above, ${\mathrm{tr}}\big({\mathbf{op}}(x^*) {\mathbf{op}}(x)\big) = 1$,
and $\big|{\mathrm{tr}}\big({\mathbf{op}}(x^*) {\mathbf{op}}(y)\big)\big| \leq 1$
for any $y \in \ball(\ell^p_n \otimes X)$
(here ${\mathbf{op}}(x), {\mathbf{op}}(y)$ are viewed as taking $\ell^{p^\prime}_n$ to $X$).
By Lemma \ref{l:compute p-norm},
$$
\|u\|_p = \max \big\{ \big|{\mathrm{tr}}\big({\mathbf{op}}(x^*) {\mathbf{op}}(y)\big)\big| :
 y \in \ball(\ell^p_n \otimes X) \big\} = 1 .
$$
%
On the other hand (see Lemma \ref{l:compute p-norm} again),
$$
\big\| (I_{\ell^p_n} \otimes u) x \big\| =
\nu_1 \big( u \circ {\mathbf{op}}(x) \big) \geq
\big| {\mathrm{tr}}\big( u \circ {\mathbf{op}}(x) \big) \big| ,
$$
hence $u$ has the desired properties.

Now suppose $X$ is the dual of some $p^\prime$-multinormed space $X_*$.
Fix $x \in \ell^p_n \otimes X$ and $\vr > 0$. There exists
$x_* \in \ball(\ell^{p^\prime}_n \otimes X_*)$ so that
$|\langle x, x_* \rangle | \geq (1-\vr) \|x\|$.
Reasoning as before, we conclude that $x_*$ determines
a weak$^*$ continuous $p$-multicontraction $u : X \to \mx_p(\ell^{p^\prime}_n)$,
with $\|(I_{\ell^p_n} \otimes u) x\| \geq (1-\vr) \|x\|$.
Taking a direct sum, we obtain the desired embedding.
\end{proof}

The preceding theorem has a ``finite dimensional'' version.

\begin{proposition}\label{p:embed f.d.}
For every $\vr > 0$ and $n \in \NN$ there exists $N = N(n,\vr) \in \NN$ so that
for any $n$-dimensional $p$-multinormed space $E$ there
exists $E^\prime \subset \ell^\infty_N(\mx_p(\ell^{p^\prime}_N))$
and a $p$-multicontraction $U : E \to E^\prime$ with
$\|U^{-1}\|_p < 1 + \vr$.
\end{proposition}

The following lemma is folklore, and can be proved by comparing volumes (see e.g. \cite{MS}).

\begin{lemma}\label{l:volume}
If $Z$ is a finite dimensional normed space, and $\delta > 0$, then
$\ball(Z)$ contains a $\delta$-net of cardinality not exceeding $(\delta^{-1}+1)^{\dim Z}$.
\end{lemma}

\begin{proof}[Proof of Proposition \ref{p:embed f.d.}]
Without loss of generality, $\vr \in (0,1)$.
Let $M = 2^n \lceil 12 n^3/\vr \rceil^n$.
Let $\delta \in \vr/(2+\vr)$. By Lemma \ref{l:volume},
we can find a $\delta$-net ${\mathcal{E}}$
in $\ball(\ell^p_M \otimes E)$, with $|{\mathcal{E}}| \leq (\delta^{-1}+1)^{Mn}$.
By the proof of Theorem \ref{t:embed inj}, for every $e \in {\mathcal{E}}$
there exists a $p$-multicontraction $u_e : E \to \mx_p(\ell^{p^\prime}_M)$
so that $\|(I_{\ell^p_M} \otimes u_e)e\| = \|e\|$.
Consider
$$
U = \oplus_{e \in {\mathcal{E}}} u_e : E \to
 \ell^\infty_{|{\mathcal{E}}|}(\mx_p(\ell^{p^\prime}_M))
\, \, \, {\textrm{  and   }} \, \, \, E^\prime = U(E) .
$$
Clearly $U$ is a complete contraction. 
It suffices to show that $\|U^{-1}\|_p < 1 + \vr$.
We have $\rank \, U^{-1} = n$, hence, by Theorem \ref{t:control},
$\|U^{-1}\|_p \leq (1+\vr/3) \|I_{\ell^p_M} \otimes U^{-1}\|$.
As $1+\vr > (1+\vr/2)(1+\vr/3)$, it suffices to show that
$\|I_{\ell^p_M} \otimes U^{-1}\| \leq 1 + \vr/2$ -- or in other words, that,
if $x \in \ell^p_M \otimes E$ with $\|x\| = 1$, then
$\|(I_{\ell^p_M} \otimes U)x\| \geq (1+\vr/2)^{-1}$.
Find $e \in {\mathcal{E}}$ so that $\|x - e\| \leq \delta$. Then
$$
\|(I_{\ell^p_M} \otimes U)x\| \geq \|(I_{\ell^p_M} \otimes U)e\| - \|U\|_p \|e-x\| \geq
1 - \delta = \frac{1}{1 + \vr/2} .
\qedhere
$$
\end{proof}

Theorem \ref{t:embed inj} implies a ``subquotient representation'' result.
We equip the Banach lattice $X = \ell^\infty(I, \ell^1)$ with its canonical $p$-multinorm.

\begin{theorem}\label{t:subquotient}
For any $p$-multinormed space $X$ there exists an index set $I$
so that $X$ is $p$-multiisometric to a subspace of a quotient of the
Banach lattice $\ell^\infty(I, \ell^1)$.
\end{theorem}

\begin{proof}
Note first that the canonical lattice $p$-multinorm on $\ell^1$ coincides with
the maximal $p$-multinorm. Indeed, denote by $(\delta_i)_{i=1}^n$ and $(e_j)_{j=1}^\infty$ the
canonical bases in $\ell^p_n$ and $\ell^1$, respectively.
A generic element of $\ell^p_n \otimes \ell^1$ can be written as
$x = \sum_{i=1}^n \delta_i \otimes x_i = \sum_j a_j \otimes e_j$,
with $x_i = (t_{ij})_{j \in \NN}$ and $a_j = (t_{ij})_{i=1}^n$. Then
$$
\big\|x\big\|_{\ell^p_n \projtens \ell^1} = \big\|x\big\|_{\ell^1(\ell^p_n)} =
\sum_j \|a_j\| = \sum_j \big( \sum_i |t_{ij}|^p \big)^{1/p} = 
\big\| \big( \sum_i |x_i|^p \big)^{1/p} \big\|_{\ell^1} .
$$
The left and right hand sides represent the maximal $p$-multinorm of $x$,
and the canonical lattice $p$-multinorm of $x$, respectively.


Now recall that, for any $i \in I$, we can find a quotient map
$q_i : \ell^1 \to \ell^{p^\prime}_{n_i}$. 
By the ``projectivity'' of the projective tensor product,
$q_i : \mx_p(\ell^1) \to \mx_p(\ell^{p^\prime}_{n_i})$ is
a $p$-quotient.

By the first paragraph of the proof,
the canonical $p$-multinorm structure of the Banach lattice $\ell^\infty(I, \ell^1)$
coincides with that of $\ell^\infty(I, \mx_p(\ell^1))$. More precisely, for
$x_i \in \ell^p_n \otimes \ell^1$ ($i \in I$), we have
$$
\big\|(x_i)_{i \in I}\big\|_{\ell^p_n \otimes \ell^\infty(I, \ell^1)} =
 \sup_{i \in I} \|x_i\|_{\ell^p_n \otimes \mx_p(\ell^1)} .
$$
Furthermore,
$$
\oplus_{i \in I} q_i : \ell^\infty(I, \ell^1) \to
 \big(\sum_{i \in I} \mx_p(\ell^{p^\prime}_{n_i}) \big)_\infty
$$
is a $p$-quotient map.
\end{proof}



Moving to the local theory of $p$-multinormed spaces, we prove:

\begin{corollary}\label{c:alm inj}
Suppose $E$ is an $n$-dimensional subspace of an infinite dimensional
$p$-multinormed space $X$. Fix $\vr > 0$, and let $N = N(n,\vr)$ be as in
Proposition \ref{p:embed f.d.}. Then there exists an $N^2$-codimensional subspace
$Y \subset X$, containing $E$, so that there exists a projection $P$ from $Y$
onto $E$, with $\|P\|_p < 1+\vr$.
\end{corollary}

\begin{proof}
By Proposition \ref{p:embed f.d.}, there exists a surjective $p$-multicontraction
$u : E \to E^\prime \subset \ell^\infty_N(\mx_p(\ell^{p^\prime}_N))$ so that $\|u^{-1}\|_p < 1 + \vr$.
By Theorem \ref{t:gen inj}, $u$ has a $p$-mul\-ti\-con\-trac\-tive extension
$\tilde{u} : E \to \ell^\infty_N(\mx_p(\ell^{p^\prime}_N))$. Let $Y = \tilde{u}^{-1}(E^\prime)$
(the preimage of $E^\prime$). Then
$\dim X/Y = \dim \ell^\infty_N(\mx_p(\ell^{p^\prime}_N))/E^\prime < N^2$.
Further, $P = u^{-1} \tilde{u}|_Y$ is a projection from $Y$ onto $E$, and
$\|P\|_p \leq \|u^{-1}\|_p \|\tilde{u}\|_p < 1 + \vr$.
\end{proof}

Taking a clue from the theory of Banach (or operator) spaces, we say that a sequence
of non-zero linearly independent elements of $(x_i)_{i=1}^\infty$ a $p$-multinormed space $X$ is a
{\emph{$p$-basic sequence}} if $\sup_n \|P_n\|_p < \infty$. Here $P_n$ is the
$n$-th basic projection: for any finite family $(\alpha_i)$, we set
$P_n (\sum_i \alpha_i x_i) = \sum_{i \leq n} \alpha_i x_i$.
Via a standard argument, Corollary \ref{c:alm inj} implies:

\begin{corollary}\label{c:basis}
Suppose $X$ is an infinite dimensional $p$-multinormed space, and $(\vr_i)_{i=1}^\infty$
is a sequence of positive numbers. Then $X$ contains a sequence of norm one
linearly independent vectors $(x_i)_{i=1}^\infty$ with the property that, for every $n$,
we have $\|P_n\|_p \leq 1 + \vr_n$, where $P_n$ is defined as above.
\end{corollary}

The corresponding Banach space result is well known. However, there exist operator
spaces without a complete basic sequence \cite{OiRi}.

\section{Projectivity}\label{s:projectivity}


\begin{proposition}\label{p:l^p_n proj}
For $1 \leq p \leq \infty$, $\mn_p(\ell^{p^\prime}_n)$ is $1$-projective as a $p$-mul\-ti\-nor\-med space.
\end{proposition}

The following auxiliary result may be of independent interest.

\begin{lemma}\label{l:lift}
For $n \in \NN$ and $\vr > 0$, there exists $M = M(n,\vr)\in \NN$
with the following property. Suppose $Y$ is a subspace of a $p$-multinormed space $X$,
and $q : X \to X/Y$ is the corresponding quotient map. If $E$ is an $n$-dimensional subspace of $X/Y$,
then there exists a subspace $F \subset X$ with $\dim F \leq M$ so that, for any $m$ and any
$e \in \ell^p_m \otimes E$, there exists $x \in \ell^p_m \otimes F$ so that
$\|x\| \leq (1+\vr)\|e\|$, and $(I_{\ell^p_m} \otimes q)x = z$.
\end{lemma}



\begin{proof}[Poof of Lemma \ref{l:lift}]
Without loss of generality, $0 < \vr < 1/2$.
Set $n = \dim E$, and $L = L(n,\vr) = 2^n \lceil 16n^3/\vr \rceil^n$.
Find a $\vr/6$-net ${\mathcal{E}}$ in the unit ball of $\ell^p_L \otimes E$,
with $|{\mathcal{E}}| \leq (1+6/\vr)^{nL}$.
For each $e \in {\mathcal{E}}$ find a lifting $z(e) = \sum_{j=1}^L \delta_j \otimes x_j(e)
 \in \ell^p_L \otimes X$, with $\|z(e)\| \leq (1+\vr/9) \|e\|$ 
(here $(\delta_j)_{j=1}^L$ denotes the canonical basis of $\ell^p_L$).
Then the dimension of $F = \spn[x_j(e) : e \in {\mathcal{E}}, 1 \leq j \leq L]$
doesn't exceed $M(n,\vr) = L (1+6/\vr)^{nL}$. We have to show that any
$e \in \ball(\ell^p_m \otimes E)$ lifts to $x \in \ell^p_m \otimes F$ with $\|x\| \leq 1+\vr$.
 
A ``telescoping sum'' argument show that any
$e \in \ball(\ell^p_M \otimes E)$ possesses a lifting
$x \in \ell^p_M \otimes F$ with $\|x\| \leq 1 + \vr/3$.
Indeed, inductively we can find a sequence $e_0, e_1, \ldots \in {\mathcal{E}}$
so that, for any $k$,
$$
\Big\|e - e_0 - \frac{\vr}{6} e_1 - \ldots - \Big( \frac{\vr}{6} \Big)^{k-1} e_{k-1}\Big\|
 \leq \Big( \frac{\vr}{6} \Big)^k .
$$
Then $e$ lifts to $\sum_{k=0}^\infty (\vr/6)^k z(e_k)$, which belongs
to $\ell^p_M \otimes F$, and has norm not exceeding
$(1+ \vr/9) (1 - \vr/6)^{-1} \leq 1 = \vr/3$.

Now fix $m \in \NN$, and consider $e \in \ball(\ell^p_m \otimes E)$.
Denote by $(u_k)_{k=1}^n$ a normalized Auerbach basis in $E$.
Write $x = \sum_{k=1}^n a_k \otimes u_k$, where $\|a_k\| \leq \|x\| \leq 1$.
Use Lemma \ref{l:l_p f.d.} to find a sublattice $Z \subset \ell^p_m$,
with the following properties:
\begin{enumerate}
\item
$K = \dim Z \leq L$, where $L$ is defined above.
\item 
$Z$ is spanned by disjoint elements of $\ell^p_m$, hence it is isometric to
$\ell^p_K$. Moreover, $Z$ is the range of a conditional expectation $P$.
\item
For every $e \in \spn[a_1, \ldots, a_n]$,
$\|e - Pe\| \leq \vr/(4n)$.
\end{enumerate}


Let $e_1 = \sum_k P a_k \otimes u_k = (P \otimes I_E) e$.
As $P$ is positive and contractive, $\|e_1\| \leq \|e\| \leq 1$.
Identifying $Z \otimes E$ with $\ell^p_K \otimes E$ (as in the proof of Lemma \ref{l:sublattice}),
we conclude that $e_1$ has a lifting
$x_1 \in Z \otimes F \subset \ell^p_m \otimes F$, with $\|x_1\| \leq 1+\vr/3$.

Now set $e_2 = \sum_k (a_k - P a_k) \otimes u_k$. Viewing $u_k$ as an element
of $\ell^p_1 \otimes E \subset \ell^p_M \otimes E$, we conclude that $u_k$ has a lifting
$\tilde{u}_k \in F$, with $\|\tilde{u}_k\| \leq 1 + \vr/3$. Then $e_2$ lifts to
$x_2 = \sum_k (a_k - Pa_k) \otimes \tilde{u}_k$, with norm
$$
\|x_2\| \leq \sum_k \|a_k - Pa_k\| \|\tilde{u}_k\| \leq n \frac{\vr}{4n} \Big( 1 + \frac{\vr}{3} \Big)
 < \frac{\vr}{2} .
$$
Thus $e$ lifts to $x = x_1 + x_2$, of norm not exceeding $\|x_1\| + \|x_2\| < 1+\vr$.
\end{proof}

\begin{proof}[Proof of Proposition \ref{p:l^p_n proj}]
We have to show that, for any pair of $p$-mul\-ti\-nor\-med spaces $Y \subset X$,
any $\vr > 0$, and any $u \in B(\ell^p_n, X/Y)$, there exists a lifting
$\tilde{u} : \ell^p_n \to X$ with $\|\tilde{u}\|_p \leq (1+\vr)\|u\|_p$.
In the case when $X$ is finite dimensional, the existence of such a lifting
(with $\|\tilde{u}\|_p = \|u\|_p$) follows by duality from Proposition \ref{p:max p'}.

If $X$ is infinite dimensional, denote the range of $u$ by $E$.
Let $q : X \to X/Y$ be the quotient map. Fix $\vr > 0$. By Lemma \ref{l:lift},
we can find a finite dimensional $F$ so that, for every $e \in \ell^p_n \otimes E$,
there exists $f \in \ell^p_n \otimes F$ with $(I_{\ell^p_n} \otimes q) f = e$,
and $\|f\| \leq (1+\vr) \|e\|$. Denote the restriction of $q$ to $F$ by $v$.

Then $v^* : E^* \to F^*$ is a contraction. Set $v^*(E^*) = G \subset F^*$,
then $v^{*-1} : G \to E^*$ has $p^\prime$-multinorm not exceeding $1+\vr$. The operator
$u^* v^{*-1} : G \to \mx_{p^\prime}(\ell^p_n)$ extends to
$w : F^* \to \mx_{p^\prime}(\ell^p_n)$, with
$\|w\|_{p^\prime} \leq \|u^*\|_{p^\prime} \|v^{*-1}\|_{p^\prime} \leq 1 + \vr$.
Then $\tilde{u} = w^*$ has the desired properties.
\end{proof}


\begin{corollary}\label{c:l^p proj}
If $(n_i)_{i \in I}$ is a family of integers, then
$(\sum_{i \in I} \mn_p(\ell^{p^\prime}_{n_i}))_1$
is $1$-projective as a $p$-multinormed space.
\end{corollary}

Next we consider liftings to the second dual. For a $p$-multinormed space $Z$
(just as in the Banach space case) we have a $p$-multiisometric embedding
$\kappa_Z : Z \to Z^{**}$ (see e.g.~\cite{DLOT}).

\begin{proposition}\label{p:double dual}
Suppose 
$Y \subset X$ are $p$-multinormed spaces, $q : X \to X/Y$ is a quotient map,
$\mu$ is a $\sigma$-finite measure, and $u : \mn_p(L^{p^\prime}(\mu)) \to X$ is $p$-multibounded.
Then there exists $\tilde{u} : \mn_p(L^{p^\prime}(\mu)) \to X^{**}$ so that $\|\tilde{u}\|_p = \|u\|_p$,
and $\kappa_{X/Y} u = q^{**} \tilde{u}$.
\end{proposition}

\begin{proof}[Sketch of a proof]
Without loss of generality assume $\|u\|_p \leq 1$.
Let ${\mathcal{F}}$ be the family of all finite dimensional sublattices $Z \subset L^p(\mu)$.
It is known that any such $Z$ is lattice isometric to $\ell^{p^\prime}_N$ for some $N = N_Z$,
and moreover, there exists a (contractive and positive) conditional expectation
$P_Z$ from $L^{p^\prime}(\mu)$ onto $Z$. By Proposition \ref{p:l^p_n proj}, for any $Z \in {\mathcal{F}}$
and $\vr > 0$ there exists $\tilde{u}_{Z\vr} : Z \to X$ which lifts $u|_Z$, and satisfies
$\|u_{Z\vr}\|_p < 1+\vr$. Let $v_{Z\vr} = \kappa_X u_{Z\vr} P_Z$, and note that
$\|v_{Z\vr}\| \leq \|v_{Z\vr}\|_p < 1+\vr$.

Consider ${\mathcal{F}} \times (0,1)$ as a net, with $(Z_1, \vr_1) \prec (Z_2, \vr_2)$ iff
$Z_1 \subset Z_2$ and $\vr_1 > \vr_2$. 
For every $Z$ and $\vr$, $v_{Z\vr} \in 2 \ball(B(L^p(\mu),X^{**}))$, and
the latter set is compact in the weak$^*$ topology
(as in the proof of Corollary \ref{c:inj Lp}, we use the identity
$B(L^p(\mu),X^{**}) = (L^{p^\prime}(\mu) \projtens X^*)^*$).
Consequently, the net ${\mathcal{F}} \times (0,1)$ has a subnet ${\mathcal{U}}$ so that
${\mathrm{weak}}^*-\lim_{\mathcal{U}} (v_{Z\vr}) = u \in B(L^p(\mu),X^{**})$ exists.
To verify that $u$ has the desired properties, follow the reasoning of Corollary \ref{c:inj Lp}.
\end{proof}

\begin{corollary}\label{c:L^p proj}
Suppose 
$Y \subset X$ are $p$-multinormed spaces, $q : X \to X/Y$ is a quotient map,
$(\mu_i)_{i \in I}$ are $\sigma$-finite measures, and
$u : (\sum_i \mn_p(L^{p^\prime}(\mu_i)))_1 \to X/Y$ is $p$-multibounded.
Then there exists $\tilde{u} : (\sum_i \mn_p(L^{p^\prime}(\mu_i)))_1 \to X^{**}$
so that $\|\tilde{u}\|_p = \|u\|_p$, and $\kappa_{X/Y} u = q^{**} \tilde{u}$.
\end{corollary}

\begin{proof}[Sketch of a proof]
Write $u = \oplus u_i$, where $u_i$ is the restriction of $u$ onto
its summand $L^{p^\prime}(\mu_i)$.
By Proposition \ref{p:double dual}, $\kappa_X u_i$ has a lifting
$\tilde{u}_i : \mn_p(L^{p^\prime}(\mu_i)) \to X^{**}$,
with $\|\tilde{u}_i\|_p = \|u_i\|_p$. Now $\tilde{u} = \oplus \tilde{u}_i$
has the desired properties.
\end{proof}

By Corollary \ref{c:gliding} below, passing to the second dual is essential here.

\begin{proposition}\label{p:quot of proj}
Suppose $X$ is a $p$-multinormed space. Then there exists a family
of integers $(n_i)_{i \in I}$, so that $X$ is a quotient of
$(\sum_{i \in I} \mn_p(\ell^{p^\prime}_{n_i}))_1$.
\end{proposition}

\begin{proof}
By Theorem \ref{t:embed inj}, there exists a family, and a weak$^*$ to weak$^*$
continuous $p^\prime$-multiisometric embedding
$X^* \to (\sum_{i \in I} \mn_{p^\prime}(\ell^p_{n_i}))_\infty$.
Now pass to the predual.
\end{proof}

A standard reasoning (cf. \cite{Ble}) shows:

\begin{proposition}\label{p:into sum}
A $p$-multinormed space $X$ is $1$-projective if and only if, for every
$\vr > 0$, there exists a family of positive integers $(n_i)_{i \in I}$,
a subspace $X^\prime \subset (\sum_i \mn_p(\ell^{p^\prime}_{n_i}))_1$,
so that:
\begin{enumerate}
\item 
There exists a $p$-multicontraction $u : X^\prime \to X$ with
$\|u^{-1}\|_p \leq 1+\vr$.
\item
$X^\prime$ is $(1+\vr)$-complemented in $(\sum_i \mn_p(\ell^{p^\prime}_{n_i}))_1$.
\end{enumerate}
\end{proposition}

A ``gliding hump'' argument immediately yields:

\begin{corollary}\label{c:gliding}
Any infinite dimensional subspace of a $1$-projective $p$-mul\-ti\-nor\-med space
contains a 
copy of $\mx_p(\ell^1)$.
Consequently, for $1 < p < \infty$, the space
$\mn_p(L^{p^\prime}(\mu))$ is $1$-projective if and only if it is finite dimensional.
\end{corollary}

{\bf Acknowledgments.}
The author wishes to thank his co-authors on \cite{DLOT} (G.~Dales, N.~Laustsen, and
V.~Troitsky), as well as A.~Helemskii and N.~Nemesh, for many stimulating conversations.
This work was partially supported by the Simons Foundation Travel Award 210060.

\end{document}